\documentclass[a4paper,11pt]{article}
\usepackage[utf8x]{inputenc}
\usepackage{amsmath, amssymb, amsthm}

\usepackage[top=2.5cm, bottom=2.5cm, left=2.5cm, right=2.5cm]{geometry}
\usepackage{color}

\numberwithin{equation}{section}

\newcommand{\R}{\mathbb{R}}
\newcommand{\F}{\mathcal{F}}
\newcommand{\N}{\mathbb{N}}
\renewcommand{\L}{\mathcal{L}}
\newcommand{\G}{\mathcal{G}}

\newcommand{\1}{\mathbf{1}}

\newcommand{\dd}{\mathrm{d}}

\theoremstyle{plain}
\newtheorem{thm}{Theorem}[section]
\newtheorem*{thm*}{Theorem}
\newtheorem{lem}[thm]{Lemma}
\newtheorem{cor}[thm]{Corollary}
\newtheorem{prop}[thm]{Proposition}
\newtheorem*{prop*}{Proposition}
\theoremstyle{definition}

\theoremstyle{remark}

\newtheorem{rmk}[thm]{Remark}

\allowdisplaybreaks

\title{Conditioned Martingales\thanks{We thank Peter Carr, Peter Imkeller, and Kostas Kardaras for their comments and suggestions. We are very grateful to Ioannis Karatzas for careful reading an earlier version of this paper.  We thank an anonymous referee for their supportive remarks. N.~P. is supported by a Ph.D. scholarship of the Berlin Mathematical School.}} 
\author{Nicolas Perkowski\thanks{Perkowsk@mathematik.hu-berlin.de}\\Institut f\"ur Mathematik\\Humboldt-Universit\"at zu Berlin \and Johannes Ruf\thanks{Johannes.Ruf@oxford-man.ox.ac.uk}\\Oxford-Man Institute of Quantitative Finance\\University of Oxford}
\date{This draft: \today}

\begin{document}
\maketitle

\begin{abstract}
   It is well known that upward conditioned Brownian motion is a three-dimensional Bessel process, and that a downward conditioned Bessel process is a Brownian motion. We give a simple proof for this result, which generalizes to any continuous local martingale and clarifies the role of finite versus infinite time in this setting. As a consequence, we can describe the law of regular diffusions that are conditioned upward or downward.
   

\end{abstract}

\section{Introduction}
We study the law $Q$ of a continuous nonnegative $P$-local martingale $X$, if conditioned never to  hit zero. The key step in our analysis is the simple observation that  the conditional measure $Q$, on the corresponding $\sigma$-algebra, is given by $(X_{T}/X_0) \dd P$, where $T$ denotes the first hitting time of either $0$ or another value $y>X_0$.  This observation relates the change of measure over an infinite time horizon (through a
conditioning argument) to the change of measure in finite time (via the Radon-Nikodym derivative $X_{T}$).

Under the conditional measure $Q$, the process $X$ diverges to $\infty$, and $1/X$ is a local martingale. This insight allows us to condition $X$ downwards, which corresponds to conditioning $1/X$ upwards and can therefore be treated with our previously developed arguments.
In the case of a diffusion it is possible to write down the dynamics of the upward conditioned process explicitly, defined via its scale function,   - and similarly for downward conditioned diffusion.
  
For example, if $X$ is a $P$-Brownian motion stopped in 0, then $X$ is a $Q$-three-dimensional Bessel process.  This connection of Brownian motion and Bessel process has been well known, at least since 
 the work of McKean \cite{McKean_1963} building on Doob \cite{Doob_1957}. 
Following McKean, several different proofs were given for this result, mostly embedding this statement in a more general result such as the one about path decompositions in Williams \cite{Williams_1974}. Most of these proofs are analytical and rely strongly on the Markov property of Brownian motion and Bessel process - or even on the fact that the transition densities are known for these processes. 

As the study of the law of upward and downward conditioned processes  has usually not been the main focus of these papers,  results have, to the best of our knowledge, not been proven in the full generality of this paper, and the underlying arguments were often only indirect. 
Our proof uses only elementary arguments, it is probabilistic, and works for continuous local martingales and certain jump processes.
We show that in finite time it is not possible to obtain a Bessel process by conditioning a Brownian motion not to hit zero and we point out that conditioning a Brownian motion upward and conditioning a Bessel process downward can be understood using the same result.

In Subsection~\ref{sec: upwd} we treat the case of upward conditioning of local martingales and in Subsection~\ref{sec: downwd} the case of downward conditioning. In Section~\ref{sec: diff} we study the implications of these results for diffusions.  In Appendix~\ref{A conditioning} we illustrate that conditioning on a nullset (such as the Brownian motion never hitting zero) is highly sensitive with respect to the approximating sequence of sets. Appendix~\ref{A proof generator} contains the slightly technical proof of Proposition~\ref{P generator}, which describes the change of dynamics of a diffusion after a change of measure. In Appendix~\ref{A jumps} we study a class of jump processes that can be treated with our methods.

\subsubsection*{Review of existing literature}
The connection of Brownian motion and the three-dimensional Bessel process has been studied in several important and celebrated papers. Most of these studies have focused  on more general statements than this connection only. To provide a complete list of references is beyond this note. In the following paragraphs, we try to give an overview of some of the most relevant and influential work in this area.

For a Markov process $X$, Doob \cite{Doob_1957} studies its $h$-transform, where $h$ denotes an excessive function such that, in particular, $h(X)$ is a supermartingale. Using $h(X)/h(X_0)$ as a Radon-Nikodym density,
a new (sub-probability) measure is constructed. Doob shows, among many other results, that, if $h$ is harmonic (and additionally ``minimal,'' as defined therein), the process $X$ converges under the new measure to the points on the extended real line where $h$ takes the value infinity. In this sense, 
changing the measure corresponds to conditioning the process to the event that $X$ converges to these points. For example, if $X$ is Brownian motion started in 1, then $h(x) = x$ is harmonic and leads to a probability measure, under which $X$, now distributed as a Bessel process, tends to infinity. Our results also yield this observation; furthermore they contain the case of non-Markovian processes $X$ that are nonnegative local martingales only. 

An analytic proof of the fact that upward conditioned Brownian motion is a three-dimensional Bessel process is given in McKean's work \cite{McKean_1963} on Brownian excursions. He shows that if $W$ is a Brownian motion started in $1$, if $B \in \F_s$, where $\F_s$ is the $\sigma$-algebra generated by $W$ up to time $s$ for some $s>0$, and if $T_0$ is the hitting time of 0, then
$
   P(W \in B | T_0 > t) \rightarrow P(X \in B)
$
as $t \uparrow \infty$, where $X$ is a three-dimensional Bessel process. The proof is  based on techniques from partial differential equations. In that article, also a path decomposition is given for excursions of Brownian motion in terms of two Bessel processes, one run forward in time, and the other one run backward. McKean already generalizes all these results to regular diffusions. 

Knight \cite{Knight_1969} computes the dynamics of Brownian motion conditioned to stay either in the interval $[-a, a]$ or $(-\infty, a]$ for some $a > 0$ and thus, derives also the Bessel dynamics. To obtain these results, Knight uses
a very astute argument based on inverting Brownian local time. He moreover illustrates the complications arising from conditioning on nullsets by providing an insightful example; we shall give another example based on a direct argument, without the necessity of any computations, in 
Appendix~\ref{A conditioning} to illustrate this point further.

In his seminal paper on path decompositions, Williams \cite{Williams_1974} shows that Brownian motion conditioned not to hit zero corresponds to the Bessel process. His results extend to diffusions and reach far beyond this observation.
For example, he shows that ``stitching'' a Brownian motion up to a certain stopping time and a three-dimensional Bessel process together yields another Bessel process.
In Pitman and Yor \cite{Pitman_Yor_1981} this approach is generalized to killed diffusions. A diffusion process is killed with constant rate and conditioned to hit infinity before the killing time. This allows the interpretation of a two-parameter Bessel process as an upward conditioned one-parameter Bessel process.

Pitman \cite{Pitman_1975} proves essentially Lemma \ref{lem: upward conditioning} of this paper in the Brownian case. This is achieved by approximating the continuous processes by random walks, which can be counted.
For the continuous case, the statement then follows by a weak convergence argument. The main result of that article is Pitman's famous theorem that $2W^*-W$ is a Bessel process if $W$ is a Brownian motion and $W^*$ its running maximum.

Baudoin \cite{Baudoin_2002} takes a different approach. Given a Brownian motion, a functional $Y$ of its path and a distribution $\nu$, Baudoin constructs a probability measure  under which  $Y$ is distributed as $\nu$. The recent monograph by Roynette and Yor \cite{RoynetteYor} studies penalizations of Brownian paths, which can be understood as a generalization of conditioned Brownian motion. Under the penalized measure, the coordinate process can have radically different behavior than under the Wiener measure. In our example it does not hit zero. In Roynette and Yor \cite{RoynetteYor} there is an example of a penalized measure under which the supremum process stays almost surely bounded.

\section{General case: continuous local martingales}  \label{S general}
Let $\Omega = C_{\mathrm{abs}} := C_\text{abs}(\R_+, [0,\infty])$ be the space of $[0,\infty]$-valued functions $\omega$ that are absorbed in 0 and $\infty$, and that are continuous on $[0,T_\infty(\omega))$, where $T_\infty(\omega)$ denotes the first hitting time of $\{\infty\}$ by $\omega$, to be specified below. Let $X$ be the coordinate process, that is, $X_t(\omega) = \omega(t)$. Define, for sake of notational simplicity, $X_\infty := \sqrt{\limsup_{t \uparrow \infty} X_t \liminf_{t \uparrow \infty} X_t}$ (with $\infty \cdot 0 := 1$).\footnote{The definition of $X_\infty$ is not further relevant as $X$ converges (or diverges to infinity) almost surely under all measures that we shall consider. We chose this definition of $X_\infty$ since it commutes with taking the reciprocal $1/X_\infty$.}  Denote the canonical filtration by $(\F_t)_{t \geq 0}$ with $\F_t = \sigma(X_s: s \le t)$, and write $\F = \vee_{t \ge 0} \F_t$. 
For all $a \in [0, \infty]$,  define $T_a$ as the first hitting time of $\{a\}$, to wit,
\begin{align}  \label{E T}
   T_a = \inf\{ t \in [0, \infty]: X_t = a\} 
\end{align} 
with $\inf \emptyset := \mathfrak{T}$, representing a time ``beyond infinity.'' The introduction of $\mathfrak{T}$ allows for a unified approach to treat examples like geometric Brownian motion. We shall extend the natural ordering to $[0,\infty] \cup \{\mathfrak{T}\}$ by $t < \mathfrak{T}$ for all $t \in[0, \infty]$. For all stopping times $\tau$, define the $\sigma$-algebras $\F_\tau $ as
\begin{align*}
   \F_\tau & = \{ A \in \F: A \cap \{\tau \le t\} \in \F_t \quad \forall t \in [0, \infty)\}  = \sigma(X^{\tau}_s: s < \infty) = \sigma(X^{\tau \wedge T_0}_s: s < \infty),
\end{align*}
where $X^{\tau} \equiv X^{\tau \wedge T_0}$  is the process $X$ stopped at the stopping time $\tau$.
Let $P$ be a probability measure on $(\Omega, \F)$, such that $X$ is a nonnegative local martingale with $P(X_0 = 1) = 1$.

\subsection{Upward conditioning}  \label{SS upward}
\label{sec: upwd}
In this section, we study the law of the local martingale $X$ conditioned never to hit zero. This event can be expressed as
\begin{align}  \label{E iden}
	 \{T_0 = \mathfrak{T}\} = \bigcap_{a \in [0, \infty)} \{T_a \le T_0\} \supset \bigcup_{a \in (0, \infty]} \{T_a \wedge T_0 = \mathfrak{T}\}.
\end{align}
The core of this article is the following simple observation:
\begin{lem}[Upward conditioning]\label{lem: upward conditioning}
   If  $P(T_a \wedge T_0 < \mathfrak{T}) =1$ for some $a \in (1, \infty)$, we have that
   \begin{align*}
      \dd  P(\cdot | T_a \le T_0) = X_{T_a} \dd  P.
   \end{align*}
\end{lem}
\begin{proof}
   Note that $X^{T_a}$ is bounded and thus a uniformly integrable martingale. In particular,
   \begin{align*}
      1 = E_P(X^{T_a}_\infty) = a P( T_a \le T_0 ) + 0,
   \end{align*}
   which implies that, for all $A \in \F$,
   \begin{align*}
      P(A|T_a \le T_0) = \frac{P(A \cap\{ T_a \le T_0\})}{P( T_a\le T_0)} = \frac{ P(A \cap\{T_a \le T_0\})}{\frac{1}{a}} = E_P\left( X^{T_a}_\infty 1_A\right),
   \end{align*}
yielding the statement.
\end{proof}

\subsubsection*{Three different probability measures}
Consider three possible probability measures:
\begin{enumerate}
	\item The local martingale $X$ introduces an \emph{$h$-transform} $Q$ of $P$. This is the unique probability measure $Q$ on $(\Omega, \F)$ that satisfies
   $
   \dd Q|_{\F_\tau} = X_\tau \dd P|_{\F_\tau}
  $
   for all stopping times $\tau$ for which $X^\tau$ is a uniformly integrable martingale. The probability measure $Q$ is called the \emph{F\"ollmer measure} of $X$, see F\"ollmer \cite{F1972} and Meyer \cite{M}.\footnote{See also Delbaen and Schachermayer \cite{DS_Bessel} for a discussion of this measure, Pal and Protter \cite{PP} for the extension to infinite time horizons and Carr, Fisher, and Ruf \cite{CFR2011} for allowing nonnegative local martingales.} 	Note that the construction of this measure does not require the density process $X$ to be the canonical process on $\Omega$ - the extension only relies on the topological structure of $\Omega = C_{\mathrm{abs}}$. This will be important later, when we consider diffusions. We remark that, in the case of $X$ being a $P$-martingale, we could also use a standard extension theorem, such as Theorem~1.3.5 in Stroock and Varadhan \cite{SV_multi}.
	\item If $P(T_0 = \mathfrak{T}) = 0$, Lemma~\ref{lem: upward conditioning} in conjunction with \eqref{E iden} directly yields the consistency of the family of probability measures
$ \{P(\cdot | T_a \le T_0)\}_{a > 1}$
on the filtration $(\F_{ T_a})_{a >1}$.  By  F\"ollmer's construction again, there exists a unique probability measure $\widetilde{Q}$ on $(\Omega, \F)$, such that $\widetilde{Q}|_{\F_{T_a}} = P(\cdot|T_a \le T_0)|_{\F_{T_a}}$. 
	\item If $P(T_0 = \mathfrak{T})  > 0$, we can define the probability measure $\widehat{Q}(\cdot) = P(\cdot | T_0 = \mathfrak{T})$ via the Radon-Nikodym derivative $1_{\{T_0 = \mathfrak{T}\}}/P(T_0 = \mathfrak{T})$. 
\end{enumerate}

Since in the case $P(T_0 = \mathfrak{T}) = 0$,  we have $\{T_a \le T_0\}  =_{P-a.s.}   \{T_a < T_0\}$ for all $a \in (0,\infty]$,
the measure $\widetilde{Q}$ is also called \emph{upward conditioned} measure since it is constructed by iteratively conditioning the process $X$ to hit any level $a$ before hitting $0$.

\subsubsection*{Relationship of probability measures}
We are now ready to relate the three probability measures constructed above:
\begin{thm}[Identity of measures]\label{thm: h-transform}
   Set $b := P(T_0= \mathfrak{T}) =  P(X_\infty > 0)$. If $b= 0$, then $Q = \widetilde{Q}$. If $b  >  0$, then $Q = \widehat{Q}$ if and only if $X$ is a uniformly integrable martingale with $P(X_\infty \in \{0,1/b\})=1$.
\end{thm}
\begin{proof}
  First, consider the case  $b= 0$.  Both $Q$ and $\widetilde{Q}$ satisfy, for all $a > 1$,
   \begin{align*}
      \dd \widetilde{Q}|_{\F_{T_a}} = X_{T_a} \dd P|_{\F_{T_a}} = \dd Q|_{\F_{T_ a}}.
   \end{align*}
   Thus $Q$ and $\widetilde{Q}$  agree on $\vee_{a >1} \F_{T_a} = \vee_{a > 1} \sigma(X^{T_a}_{t}: t \ge 0)= \F$.

	Next, consider the case  $b  > 0$. Then,  $Q = \widehat{Q}$ and  $\dd \widehat{Q} / \dd  P|_{\F_t} \leq 1/b$ imply that $X_t \leq 1/b$, yielding that $X$ is a uniformly integrable martingale with $X_\infty = \dd Q / \dd P  \in \{0,1/b\}$. For the reverse direction,   observe that
	 $X_\infty = 1_{\{T_0 = \mathfrak{T}\}}/b$. This observation together with its uniform integrability completes the proof.
\end{proof}

This theorem implies, in particular, that in finite time the three-dimensional Bessel process cannot be obtained from conditioning Brownian motion not to hit zero.  However, over finite time-horizons, a Bessel-process can be constructed via the $h$-transform $X_T \dd P$, when $X$ is $P$-Brownian motion started in 1 and stopped in 0. Over infinite time-horizons, one has two choices; the first one is using an extension theorem for the $h$-transforms, the second one is
conditioning $X$ not to hit 0 by approximating this nullset by the sequence of events that $X$ hits any $a>0$ before it hits $0$.

\begin{rmk}[Conditioning on nullsets]
   We remark that the interpretation of the measure $\widetilde{Q}$ as $P$ conditioned on a nullset requires specifying an approximating sequence of that nullset. In Appendix~\ref{A conditioning} we illustrate this subtle but important point.  \qed
\end{rmk}

\begin{rmk}[The trans-infinite time $\mathfrak{T}$]
	The introduction of $\mathfrak{T}$ in this subsection allows us to introduce the upward-conditioned measure $\widetilde{Q}$ and to show its equivalence to the $h$-transform $Q$ if $X$ converges to zero but not necessarily hits zero in 
finite time, such as $P$-geometric Brownian motion.  If one is only interested in processes as, say, stopped Brownian motion, then one could formulate all results in this subsection in the standard way when $\inf \emptyset := \infty$ in \eqref{E T}. One would then need to exchange $\mathfrak{T}$ by $\infty$ throughout this subsection; in particular, one would have to assume in Lemma~\ref{lem: upward conditioning} that $P(T_a \wedge T_0 < \infty) =1$ and replace the condition $P(T_0 = \mathfrak{T}) = 0$ by
$P(T_0 = \infty) = 0$ for the construction of the upward-conditioned measure $\widetilde{Q}$.  \qed
\end{rmk}

We note that the arguments of this section can be extended to certain jump processes. In Appendix~\ref{A jumps} we treat a simple random walk example to illustrate this observation.

\subsection{Downward conditioning}
\label{sec: downwd}
In this subsection, we consider the converse case of conditioning $X$ downward instead of upward. Towards this end, we first provide a well-known
 result; see for example \cite{CFR2011}. For sake of completeness, we provide a direct proof:
\begin{lem}[Local martingality of $1/X$]  \label{C 1/x}
	Under the $h$-transformed measure $Q$, the process $1/X$ is a nonnegative local martingale and $Q(T_\infty = \mathfrak{T}) = E_P[X_\infty]$.
\end{lem}
\begin{proof}
	Observe that
\begin{align*}f
   E_Q\left(1_{A} \frac{1}{X^{T_{1/n}}_{t+s}}\right) &=  \lim_{m \uparrow \infty} E_Q\left(1_{A \cap \{T_{m} > t\}} \frac{1}{X^{ T_{1/n} \wedge T_m}_{t+s}}\right)  +  E_Q\left(1_{A \cap \{T_{\infty} \le t\}} \frac{1}{X^{ T_{1/n}}_{t+s}}\right) \\
	&=  \lim_{m \uparrow \infty}  E_P\left(1_{A \cap \{T_{m} > t\}} \frac{1}{X^{ T_{1/n} \wedge T_m}_{t+s}}  X^{T_{m}}_{t+s}\right) +  E_Q\left(1_{A \cap \{T_{\infty} \le t\}} \frac{1}{X^{ T_{1/n}}_{t}}\right)\\
	& = \lim_{m \uparrow \infty}  E_P\left(1_{A \cap \{T_{m} > t\}} \frac{1}{X^{ T_{1/n} \wedge T_m}_{t}}  X^{T_{m}}_{t}\right) +  E_Q\left(1_{A \cap \{T_{\infty} \le t\}} \frac{1}{X^{ T_{1/n}}_{t}}\right) \\
	& = \lim_{m \uparrow \infty}  E_Q\left(1_{A \cap \{T_{m} > t\}} \frac{1}{X^{ T_{1/n} \wedge T_m}_{t}}\right) +  E_Q\left(1_{A \cap \{T_{\infty} \le t\}} \frac{1}{X^{ T_{1/n}}_{t}}\right) \\
	&=    E_Q\left(1_{A} \frac{1}{X^{T_{1/n}}_{t}}\right) 
\end{align*}
for all $A \in \F_t$ and $s,t \ge 0$, where in the third equality we considered the two events $\{T_{1/n} \leq t\}$ and $\{ T_{1/n}>t\}$ separately and used the $P$-martingality of $X^{T_m}$ after conditioning on $\F_t$ and $\F_{T_{1/n}}$, respectively - note that $A\cap\{T_m > t\} \cap \{ T_{1/n}>t\} \in \F_{T_{1/n}}$.

The local martingality of $1/X$ then follows from 
\begin{align*}
	Q\left(\lim_{n \rightarrow \infty}T_{1/n} < \infty\right) = \lim_{m \uparrow \infty} Q\left(\lim_{n \rightarrow \infty}T_{1/n} < T_m \wedge \infty\right)=  \lim_{m \uparrow \infty} E_P\left(1_{\{\lim_{n \rightarrow \infty}T_{1/n} < T_m\}} X^{T_m}_\infty\right) = 0 .  
\end{align*}
Therefore, $1/X$  converges $Q$-almost surely to some random variable $1/X_\infty$. We observe that 
\begin{align*}
	Q(T_\infty = \mathfrak{T}) &= 1- \lim_{m \uparrow \infty} Q(T_m < \infty)= 1 - \lim_{m \uparrow \infty} E_P(1_{\{T_m < \infty\}} X^{T_m}_\infty) \\
	&=  \lim_{m \uparrow \infty} E_P(1_{\{T_m \geq \infty\}} X_\infty) = E_P(X_\infty),
\end{align*}
where we use that $X$ converges $P$-almost surely.
\end{proof}

The last lemma directly implies the following observation:
\begin{cor}[Mutual singularity] \label{C mutual}
	We have $P(X_\infty = 0) = 1$ if and only if  $Q(X_\infty = \infty) =1$.
\end{cor}
This observation is consistent with our understanding that either condition implies that the two measures are supported on two disjoint sets. 
Corollary~\ref{C mutual} is also consistent with Theorem~\ref{thm: h-transform}, which yields that $P(X_\infty =0) =1$ implies the identity $Q=\widetilde{Q}$, where $\widetilde{Q}$ denotes the  upward conditioned measure.

Lemma~\ref{C 1/x} indicates that we can condition $X$ \emph{downward} under $Q$, 
corresponding to conditioning $1/X$ \emph{upward}. The proof of the next result is exactly along the lines of the arguments in Subsection~\ref{SS upward}; however, 
now with the $Q$-local martingale $1/X$ taking the place of the $P$-local martingale $X$:

\begin{thm}[Downward conditioning]\label{thm: downward conditioning}
   If $b$ of Thereom~\ref{thm: h-transform} satisfies $b=0$, then
   \begin{align*}
      \dd Q(\cdot | T_{1/a} \leq T_\infty) = \frac{1}{X_{T_{1/a}}} \dd Q
   \end{align*}
    for all $a>1$.
   In particular, there exists a unique probability measure $\widetilde{P}$, such that $\widetilde{P}|_{\F_{T_{1/a}}} = Q(\cdot|T_{1/a} < \mathfrak{T})$; in fact, $\widetilde{P} = P$.
\end{thm}

\section{Diffusions}
\label{sec: diff}
In this section, we apply Theorems~\ref{thm: h-transform} and \ref{thm: downward conditioning} to diffusions.

\subsection{Definition and $h$-transform for diffusions}
We call \emph{diffusion} any time-homogeneous strong Markov process $Y: C_{\mathrm{abs}}\times[0, \infty) \rightarrow [l,r]$ with continuous paths in a possibly infinite interval $[l,r]$ with $-\infty \leq l < r \leq \infty$. Note that we explicitly allow $Y$ to take the values $l$ and $r$;
we stop $Y$ once it hits the boundary of $[l,r]$.
We define $\tau_a$ for all $a \in [l,r]$ as in \eqref{E T} with $X$ replaced by $Y$. 
We denote the probability measure under which $Y_0 = y \in [l,r]$ by $P_y$. 

Since $Y$ is Markovian it has an \emph{infinitesimal generator} (see page 161 in Ethier and Kurtz \cite{EK_Markovian}). 
As we do not assume any regularity of the semigroup of $Y$, we find it convenient to work with the following \emph{extended infinitesimal generator}:
A continuous function $f:[l,r] \rightarrow \R \cup\{-\infty, \infty\}$ with $f|_\R \in \R$ is in the \emph{domain} of the extended infinitesimal generator $\L$ of $Y$ if there exists a continuous function $g:[l,r] \rightarrow \R \cup\{-\infty, \infty\}$ with $g|_\R \in \R$, and an increasing sequence of stopping times $\{\rho_n\}_{n \in \N}$, such that $P_y(\lim_{n \rightarrow \infty} \rho_n \ge \tau_l \wedge \tau_r) = 1$ and 
	\begin{align*}
                f(Y^{\rho_n}_\cdot) - f(y) - \int_0^{\cdot\wedge \rho_n} g(Y_s) \dd s
             \end{align*}
             is a $P_y$-martingale for all $y \in (l,r)$.
In that case we write $f \in \mathrm{dom}(\L)$ and $\L f = g$.

Throughout this section we shall work with a  \emph{regular} diffusion $Y$; that is, 
 for all $y,z \in (l,r)$ we have that $P_y(\tau_z < \infty) > 0$.  In that case there always exists a continuous, strictly increasing function $s:  (l,r) \rightarrow \R\cup\{-\infty, \infty\}$, uniquely determined up to an affine transformation, such that $s(Y)$ is a local martingale (see Propositions~VII.3.2 and VII.3.5 in Revuz and Yor \cite{RY}). We call every such $s$ a \emph{scale function} for $Y$, and we extend its domain to $[l,r]$ by taking limits.
 The next result summarizes Proposition~VII.3.2 in \cite{RY} and describes the relationship of the scale function $s$ and the limiting behaviour of $Y$:

\begin{lem}[Scale function]\label{lem: diff bdry bhv} We have that
\begin{enumerate}
	\item $P_y(\tau_l = \mathfrak{T}) = 0$ for one (and then for all) $y \in (l,r)$ if and only if $s(l) \in \R$ and $s(r) = \infty$;
	\item $P_y(\tau_r = \mathfrak{T}) = 0$ for one (and then for all) $y \in (l,r)$ if and only if $s(l) = - \infty$ and $s(r) \in \R$;
	\item $P_y(\tau_l \wedge \tau_r = \mathfrak{T}) = 0$ and $P_y(\tau_l < \mathfrak{T}) \in (0,1)$ for one (and then for all) $y \in (l,r)$ if and only if $s(l)  \in \R$ and $s(r) \in \R$.
\end{enumerate}
\end{lem}

Throughout this section, we shall work with the standing assumption that 
the scale function $s$ satisfies $s(l) > -\infty$ (Assumption~L) or $s(r) < \infty$ (Assumption~R). Without loss of generality, we shall assume that then $s(l) = 0$ or $s(r) = 0$, respectively, and that $\F = \F_{\tau_l \wedge \tau_r}$.

Since by assumption $s(Y)$ is a local martingale, it defines, under each $P_y$, a F\"ollmer measure $Q_y$ as in Section~\ref{S general}, where we would set $X := s(Y)/s(y)$, for all $y \in [l,r]$ (with $0/0 :=\infty / \infty := 1$).  The next proposition illustrates how the extended infinitesimal generators of $Y$ under $P_y$ and $Q_y$ are related:
\begin{prop}[$h$-transform for diffusions]  \label{P generator}
	The process $Y$ is  a regular diffusion under the probability measures $\{Q_y\}_{y \in [l,r]}$. Its extended infinitesimal generator $\L^s$ under $\{Q_y\}_{y \in [l,r]}$ is given by $\mathrm{dom}(\L^s) = \{\varphi: s \varphi \in \mathrm{dom}(\L)\}$ and 
   \begin{align*} 
      \L^{s} \varphi(y) = \frac{1}{s(y)} \L[s \varphi](y).
\end{align*}  
\end{prop}
The proof of this proposition is technical and therefore postponed to Appendix~\ref{A proof generator}. 
The following observation is a direct consequence of Lemma~\ref{C 1/x} and the fact that $Y$ is a regular diffusion under the probability measures $\{Q_y\}_{y \in [l,r]}$:
\begin{lem}[Scale function for $h$-transform]  \label{L 1/s}
	Under $\{Q_y\}_{y \in [l,r]}$, the function $\widetilde{s}(\cdot) = -1/s(\cdot)$ is, with the appropriate definition of $1/0$, a scale function for $Y$ with $\widetilde{s}(l) = -\infty$, $\widetilde{s}(r) \in \R$
	under Assumption~L and with $\widetilde{s}(r) = \infty$, $\widetilde{s}(l) \in \R$
	 under Assumption~R.
\end{lem}

\subsection{Conditioned diffusions}
We now are ready to formulate and prove a version of the statements of Section~\ref{S general} for diffusions:
\begin{cor}[Conditioning of diffusions]\label{cor: diffusion upward}
    Fix $y \in (l,r)$ and make Assumption~L.
	\begin{enumerate}
		\item Suppose that $P_y(\tau_l =\mathfrak{T})= 0$, which is equivalent to $s(r) = \infty$.  Then the family of probability measures $\{P_y( \cdot | \tau_a \le \tau_l)|_{\F_{\tau_a}}\}_{y<a<r}$ is consistent and thus has an
extension $\widetilde{Q}_y$ on $\F$. Moreover, the extension satisfies $\widetilde{Q}_y = Q_y$.
		\item Suppose that $P_y(\tau_l = \mathfrak{T})  > 0$, which is equivalent to $s(r) < \infty$, and define $\widehat{Q}_y = P_y(\cdot|\tau_l = \mathfrak{T})$. Then $\widehat{Q}_y$ satisfies $\widehat{Q}_y = Q_y$.
	\end{enumerate}

	Furthermore,  provided that $s(r) = \infty$, the family of probability measures $\{Q_y( \cdot| \tau_a \le \tau_r)|_{\F_{\tau_a}}\}_{l<a<y}$ is consistent. Its unique extension is $P_y$.

	Under Assumption~R, all statements still hold with $r$ exchanged by $l$ and, implicitly, $y<a<r$ exchanged by $l < a < y$. 
\end{cor}
\begin{proof} We only consider the case of Assumption~L, as Assumption~R requires the same steps.
   We write $X = s(Y)/s(y)$. The hitting times $T_a$ of $X$ are defined as in \eqref{E T}.   Since $s$ is strictly increasing, we have that, for all $y < a < r$,
   \begin{align*}
      \{\tau_a \le \tau_l\} = \{T_{s(a)/s(y)} \le T_0\}.
   \end{align*}
   Since $X$ is a nonnegative local martingale with $P_y(X_0=1)=1$, the statements in 1.~and 2.~follow immediately from Theorem~\ref{thm: h-transform} and Lemma~\ref{lem: diff bdry bhv}, which shows that $s(Y)_\infty$ takes exactly
two values.  The remaining assertions follow from Lemma~\ref{L 1/s} and Theorem~\ref{thm: downward conditioning}.
\end{proof}

It is clear that the measure $Q$ under Assumption~L corresponds to the upward conditioned diffusion $Y$, while under Assumption~R it corresponds to the downward conditioned diffusion.

 After finishing this manuscript we learned about Kardaras \cite{Kardaras_times}. Therein, by similar techniques it is shown that $Y$ under $Q$ tends to infinity if $s(r) = \infty$; see Section~6.2 in \cite{Kardaras_times}.
 In Section~5 therein, a similar probability measure is constructed for a L\'evy process $X$ that drifts to $-\infty$. After a change of measure of the form $s(X)$ for a harmonic function $s$, the process $X$ under the new measure
drifts now again to infinity.

\subsection{Explicit generators}
In this section we formally derive the dynamics of upward conditioned and downward conditioned diffusions. For this purpose suppose that $Y$ is a diffusion with extended infinitesimal generator $\L$, such that $\mathrm{dom}(\L) \supseteq C^2$, where $C^2$ denotes the space of twice continuously differentiable functions on $(l,r)$, and
\begin{align*}
   \L \varphi(y) = b(y) \varphi'(y) + \frac{1}{2} a(y) \varphi''(y), \quad \varphi \in C^2
\end{align*}
for some locally bounded functions $b$ and $a$ such that $a(y) > 0$ for all $y \in (l,r)$.

Finding the scale function then at least formally corresponds to solving the linear ordinary differential equation
\begin{align}  \label{E ode s}
   b(y) s'(y) + \frac{1}{2} a(y) s''(y) = 0.
\end{align}
This is for example done in Section~5.5.B of Karatzas and Shreve \cite{KS1}.  From now on, we continue under either Assumption~L or Assumption~R with $s$ being either nonnegative or nonpositive.
We plug  $s$ into the definition of $\L^s$. 
Towards this end, let $\varphi \in C^2$. Then we have that
\begin{align*}
   \L^s \varphi(y) & = \frac{1}{s(y)} \L (s \varphi)(y) = \frac{1}{s(y)} \left( b(y) (s\varphi)'(y) + \frac{1}{2} a(y) (s\varphi)''(y)\right) \\
   & = \frac{1}{s(y)} \left( b(y) (s'(y)\varphi (y)+ s(y) \varphi'(y)) + \frac{1}{2} a(y) (s''(y)\varphi(y) + 2s'(y)\varphi'(y) + s(y)\varphi''(y))\right) \\
   & =  \left(b(y) + \frac{a(y) s'(y)}{s(y)}\right) \varphi'(y) + \frac{1}{2} a(y) \varphi''(y)
\end{align*}
since $s'' = -2(b/a) s'$ due to \eqref{E ode s}. Therefore, the upward or downward conditioned process has an additional drift of $(as')/s$. This drift is always positive (or always negative), as is to be expected. 

Now, under Assumption~L (upward conditioning) with $l=0$, if $b = 0$, then $s(y) = y$; therefore the additional drift of the upward conditioned process is $a(y)/y$. 
Under Assumption~R (downward conditioning) with $r=\infty$, if $b(y) = a(y)/y$, then \eqref{E ode s} yields
 $s(y) = -\frac{1}{y}$ and thus an additional drift of $-a(y)/y = -b(y)$. These observations lead to the well-known fact:
\begin{cor}[(Geometric) Brownian motion]
   A Brownian motion conditioned on hitting $\infty$ before hitting $0$ is a three-dimensional Bessel process.  
Vice versa, a three-dimensional Bessel process conditioned to hit 0 is a Brownian motion.
   Moreover, a geometric Brownian motion conditioned on hitting $\infty$ before hitting 0 is a geometric Brownian motion with unit drift.
\end{cor}

\appendix
\section{Conditioning on nullsets}    \label{A conditioning}
Before Theorem~\ref{thm: h-transform}, we constructed a probability measure $\widetilde{Q}$ by conditioning $P$ on the nullset $\{T_0 = \mathfrak{T}\} = \bigcap_{a\in [0, \infty)} \{T_a \le T_0\}$  using an extension theorem. 
It is important to point out that the choice of the approximating sequence of events, necessary for this construction, is highly relevant.  We remark that this has been illustrated before by
Knight \cite{Knight_1969} with another example, which, in our opinion, is slightly more involved than the one presented in the following.

To illustrate the issue, consider the continuous martingale $\widetilde{X}$, defined as 
\begin{align*}
	\widetilde{X}_t = X_t + (X_t - 1) \1_{\{T_{3/4}\geq t\}} +  \left(\frac{1}{8} - \frac{X_t}{2}\right)\1_{\{T_{3/4} < t \leq T_{1/4}\}};
\end{align*}
the  process $\widetilde{X}$ moves twice as much as $X$ until $X$ hits $3/4$, then it moves half as much as $X$ until $X$ catches up, which occurs when $X$ hits $1/4$. With this understanding, it is clear that $\widetilde{X}$ hits zero exactly when $X$ hits zero. Therefore,  we have that $\{T_0 = \mathfrak{T}\} = \bigcap_{a\in [0, \infty)}  \{\widetilde{T}_a \le \widetilde{T}_0\}$, where $\widetilde{T}_a$ is defined exactly like $T_a$ with $X$ replaced by $\widetilde{X}$ in \eqref{E T}.

Now, it is easy to see that $P(\cdot|\widetilde{T}_a \le T_0)$ defines a consistent family of probability measures  on the filtration $(\F_{T_0 \wedge \widetilde{T}_a})_{a >1}$; namely the one defined through the Radon-Nikodym derivatives $\widetilde{X}_{T_a}$. Since $P(\widetilde{X}_{T_a} \neq X_{T_a}) > 0$, the induced measure differs from the one in Theorem~\ref{thm: h-transform}. Therefore, although in the limit we condition on the same event, the induced probability measures strongly depend  on the approximating sequence of events.

\section{Proof of Proposition~\ref{P generator}}  \label{A proof generator}

We only discuss the case $s(l) = 0$ since the case $s(r)=0$ follows in the same way. 
In order to show the Markov property of $Y$ under $Q_y$, we need to prove that
\begin{align*}
	      E_{Q_y}(f(Y_{\rho + t})|\F_{\rho}) = E_{Q_y}(f(Y_{\rho + t})|Y_{\rho})
\end{align*}
for all $t \geq 0$, for all bounded and continuous functions $f: [l,r] \rightarrow \R$, and for all finite stopping times  $\rho$.  On the event $\{\rho \geq \tau_r\}$, the equality holds trivially as $Y$ gets absorbed in $l$ and $r$. On the event $\{\rho < \tau_r\}$, observe that
\begin{align*}
	E_{Q_y}(f(Y_{\rho + t})|\F_{\rho}) &=  \lim_{a \uparrow r} E_{Q_y}(f(Y^{\tau_a}_{\rho + t})|\F_{\rho})  =  \lim_{a \uparrow r} E_{Q_y}(f(Y^{\tau_a}_{\rho + t})|Y^{\tau_a}_{\rho}) 
		= E_{Q_y}(f(Y_{\rho + t})|Y_{\rho}),
\end{align*}
where the second equality follows from the generalized Bayes' formula in Proposition~C.2 in \cite{CFR2011} and the Markov property of $Y^{\tau_a}$ under $P_y$.  Therefore, $Y$ is strongly Markovian under $Q_y$.
Since $Y$ is also time-homogeneous under any of the measures $Q_y$, we have shown that $Y$ is a diffusion under
$\{Q_y\}_{y \in [l,r]}$.

As for the regularity, fix $a \in (l,y)$ and $b \in (y,r)$.  Observe that  $Q_y$ is equivalent to $P_y$ on $\F_{\tau_a \wedge \tau_b}$.  This fact in conjunction with the regularity of $Y$ under $P$ and Proposition~VII.3.2 in \cite{RY} yields that $Q_y(\tau_a<\infty)>0$ as well as $Q_y(\tau_b < \infty)>0$.

   Denote now the extended infinitesimal generator of $Y$ under $\{Q_y\}_{y \in [l,r]}$ by $\G$, let $\varphi \in \mathrm{dom}(\G)$ with localizing sequence $\{\rho_n\}_{n \in \mathbb{N}}$, and fix $y \in (l,r)$. Fix two sequences $\{a_n\}_{n \in \N}$ and  $\{b_n\}_{n \in \N}$ with $a_n \downarrow l$ and $b_n \uparrow r$ as $n \uparrow \infty$. We may assume, without loss of generality, that $\rho_n \leq \tau_{a_n} \wedge \tau_{b_n}$. By definition of the extended infinitesimal generator,
   \begin{align*}
      \varphi(Y^{\rho_n}_\cdot) - \varphi(y) - \int_0^{\cdot\wedge \rho_n} \G\varphi(Y_s) \dd s
   \end{align*}
   is a $Q_y$-martingale.
   Since $\varphi(\cdot)$ and $\G\varphi(\cdot)$ are bounded on $[a_n, b_n]$ this fact, in conjunction with Fubini's theorem, yields that 
   \begin{align*}
      \frac{1}{s(y)}\left(\varphi(Y^{{\rho}_n}_\cdot)s(Y^{{\rho}_n}_\cdot) - \varphi(y)s(y) - \int_0^{\cdot \wedge {\rho}_n}  \G\varphi(Y^{{\rho}_n}_u) s(Y^{{\rho}_n}_u) \dd u \right)
   \end{align*}
   is a $P_y$-martingale.  Since $\{\rho_n\}_{n \in \mathbb{N}}$ converges $P_y$-almost surely to $\tau_l \wedge \tau_r$ for all $y \in (l,r)$ this implies that $\varphi s \in \mathrm{dom}(\L)$ and $\L[s\varphi](y) = \G \varphi(y) s(y) $. The other inclusion can be shown in the same manner, which completes the proof. \qed

\section{Jumps}  \label{A jumps}
Here we illustrate on a simple example that our results about upward conditioning can be extended to certain jump processes. Towards this end, we consider the canonical space of paths $\omega$ taking values in $[0, \infty]$, getting absorbed in either 0 or $\infty$, and being c\`adl\`ag on $[0, T_\infty(\omega))$. The measure $P$ is chosen in such a way that the canonical process $X$ is a purely discontinuous martingale starting in $1$, whose semimartingale characteristics under the truncation function $h(x) = x 1_{|x|\le 1}$ are given by $(0,0, \nu)$. Here $\nu$ is a predictable random measure, the compensator of the jump measure of $X$. We assume that 
\begin{align*}
   \nu(\omega, \dd s, \dd x) = \nu_{\mathrm{int}}(\omega, \dd s)\frac{1}{2} (\delta_{-1/N} + \delta_{1/N})(\dd x),
\end{align*}
for some $N \in \mathbb{N}$, where $\nu_{\mathrm{int}}$ denotes the jump intensity,  and that
\begin{align*}
   \nu_{\mathrm{int}}(\omega, \dd s) \ll X_{s-}(\omega) \nu_{\mathrm{int}}(\omega, \dd s);
\end{align*}
to wit, $X$ only has jumps of size $\pm 1/N$ and gets absorbed when hitting 0.
We furthermore assume
that $\nu$ is bounded away from $\infty$ and 0; that is, that for all $t \ge 0$ there exist two nonnegative functions $c(t)$ and $C(t)$ tending to infinity as $t$ increases such that 
\begin{align*}
   1_{\{X_{t-}(\omega) > 0\}}c(t) \leq \int_{[0,t]} \nu_{\mathrm{int}}(\omega, \dd s) \le C(t).
\end{align*}
For example, $X$ could be a compound Poisson process with jumps of size $\pm 1/N$, getting 
absorbed in 0. 

The conditions on $X$ guarantee that $P(T_0 < \infty) = 1$ since a one-dimensional random walk is recurrent; furthermore,  $X$ satisfies $P(T_{n/N} < \infty)>0$ for all $n \in \N$. Therefore, the assertion  of Lemma~\ref{lem: upward conditioning} holds for $a = n/N$ for all $n \in \mathbb{N}$ with $n \ge N$; hence, the $h$-transform $Q$, defined by $\dd Q|_{\F_t} = X_t\dd P|_{\F_t}$, equals the upward conditioned measure $\widetilde{Q}$, defined as the extension of the measures $\{P(\cdot|T_{n/N} \leq T_0)\}_{n \ge N}$.

Girsanov's theorem (Theorem~III.3.24 in Jacod and Shiryaev \cite{JacodS}) implies that, under the probability measure $Q = \widetilde{Q}$, the process $X$ has semimartingale characteristics $(0,0, \nu')$, where
\begin{align*}
   \nu'(\omega, \dd s, \dd x) = \nu_{\mathrm{int}}(\omega, \dd s) \frac{1}{2}\left( \frac{X_{s-}(\omega) - \frac{1}{N}}{X_{s-}(\omega)} \delta_{-1/N} + \frac{X_{s-}(\omega) + \frac{1}{N}}{X_{s-}(\omega)} \delta_{1/N}\right)(\dd x).
\end{align*}
These computations show that we cannot expect $1/X$ to be a $Q$-local martingale; indeed, in our example, the process $1/X$ is bounded by $N$ and a true $Q$-supermartingale. Thus, we cannot obtain $P$ through conditioning 
$X$ downward as we did for the continuous case in Subsection~\ref{sec: downwd}.  

Consider now the case of deterministic jump times with
\begin{align*}
   \nu_{\mathrm{int}}(\omega, \dd s) = \sum_{n=1}^{\infty} 1_{\{X_{s-}(\omega) > 0\}} \delta_{n \delta t} (\dd s),
\end{align*} where $\delta t := 1/N^2$. With a slight misuse of notation  allowing $X_0$ to take the value $x = n/N$ for some $n \in \mathbb{N}$, observe that, for all $C^2$-functions $f$, 
\begin{align*}
   \frac{1}{\delta t} \left( E_Q[f(X_{\delta t})| X_0 = x] - f(x) \right) & = \frac{1}{\delta t} \left( E_P\left[\left.f(X_{\delta t})\frac{X_{\delta_t}}{x}\right| X_0 = x\right] - f(x) \right) \\
   & = N^2 \left[ \frac{1}{2}f\left(x+\frac{1}{N}\right)\frac{x + 1/N}{x} + \frac{1}{2} f\left(x - \frac{1}{N}\right)\frac{x- 1/N}{x} - f(x) \right] \\
   & = \frac{1}{2} N^2 \left[f\left(x+\frac{1}{N}\right) + f\left(x-\frac{1}{N}\right) - 2f(x)\right] \\
   & \qquad + \frac{1}{x} \cdot \frac{N}{2}\left[f\left(x+\frac{1}{N}\right) + f\left(x-\frac{1}{N}\right) \right] \\
   & \simeq \frac{1}{2} f''(x) + \frac{1}{x} f'(x).
\end{align*}
Using arguments based on the martingale problem, we obtain the weak convergence of $X$ under $Q$ to a Bessel process as $N$ tends to infinity (see Corollary~4.8.9 in \cite{EK_Markovian}). On the other side, Donsker's theorem implies that $X$ converges weakly to a Brownian motion under $P$. We thus recover Pitman's proof that upward conditioned Brownian motion is a Bessel process; see Pitman~\cite{Pitman_1975}.

\bibliography{aa_bib}{}
\bibliographystyle{amsplain}

\end{document}